\documentclass[12pt]{article}
\usepackage{amsfonts, amsmath, amssymb,latexsym,epsfig}
\usepackage{amsthm}
\usepackage{tikz}
\usetikzlibrary{graphs}
\usetikzlibrary{graphs.standard}
\usepackage{relsize}
\usepackage{verbatim}
\usepackage{enumerate}

\newtheorem{thm}{Theorem}[section]
\newtheorem{cor}[thm]{Corollary}
\newtheorem{lem}[thm]{Lemma}
\newtheorem{prop}[thm]{Proposition}
\newtheorem{conj}[thm]{Conjecture}

\newtheorem{pro}[thm]{Problem}

\newcommand{\R}{\mathbb R}

\newcommand{\Q}{\mathbb Q}

\newcommand{\F}{\mathbb F}

\newcommand{\hs}{\hspace{1cm}}

\begin{document}

\title{\Large {{\bf A New Family of Algebraically Defined Graphs With Small Automorphism Group}}}
\bigskip

\bigskip

\author{{\bf Felix Lazebnik  and  Vladislav Taranchuk}\\\\
\ Department of Mathematical Sciences\\
  University of Delaware, Newark, DE 19716, USA \\
 {\tt fellaz@udel.edu}, {\tt
vladtar@udel.edu}}

\maketitle

\vspace{-0.2in}

\begin{abstract}  {\small Let $p$ be an odd  prime, $q=p^e$, $e\ge 1$, and $\F = \F_q$ denote the finite field of $q$ elements.  Let $f: \F^2\to \F$ and  $g: \F^3\to \F$  be functions, and  let $P$ and $L$ be two copies of the 3-dimensional vector space $\F^3$.
Consider a bipartite graph $\Gamma _\F (f, g)$ with vertex partitions $P$ and $L$ and with edges defined as follows: for every $(p)=(p_1,p_2,p_3)\in P$ and every $[l]= [l_1,l_2,l_3]\in L$, $\{(p), [l]\} = (p)[l]$ is an edge in $\Gamma _\F (f, g)$ if
$$p_2+l_2 =f(p_1,l_1) \;\;\;\text{and}\;\;\;
     p_3 + l_3 = g(p_1,p_2,l_1).$$
The following question  appeared in Nassau \cite{nassau2020}:
Given $\Gamma _\F (f, g)$,  is it always possible to find a function $h:\F^2\to \F$ such that the graph $\Gamma _\F (f, h)$  with the same vertex set as $\Gamma _\F (f, g)$ and with edges $(p)[l]$  defined in a similar way  by the system $$p_2+l_2 =f(p_1,l_1) \;\;\;\text{and}\;\;\;
     p_3 + l_3 = h(p_1,l_1),$$
is isomorphic to $\Gamma _\F (f, g)$ for infinitely many $q$?  In this paper we show that the  answer to the question is negative and the graphs $\Gamma_{\F_p}(p_1\ell_1, p_1\ell_1p_2(p_1 + p_2 + p_1p_2))$ provide such an example for $p \equiv 1 \pmod{3}$. Our argument is based on proving that the automorphism group of these graphs has order $p$, which is the smallest possible order of the automorphism group of graphs of the form $\Gamma_{\F}(f, g)$.
}
\end{abstract}

\bigskip

\section{Introduction} \label{sec1}

For a set $A$, let $|A|$ denote the cardinality of $A$. Let $p$ be an odd  prime, $q=p^e$, $e\ge 1$, $\F = \F_q$ denote the finite field of $q$ elements, $\F^\times = \F\setminus \{0\}$,  and $\F[X]$, $\F[X,Y]$, $\F[X, Y,Z]$ denote the rings of polynomials of indeterminates $X, Y,Z$  over $\F$.  A polynomial  $f\in \F[X]$, defines a polynomial function $\F\to \F$ via $a\mapsto f(a)$. We will use the same notation $f$ for this function, and write $f=f(x)$.   It is well known that over $\F$, every function can be represented uniquely by a polynomial function of degree at most $q-1$, e.g., one can use the Lagrange interpolating polynomial for this.    All undefined terms and facts that we use without proofs related to finite fields can be found in Lidl and Niederreiter \cite{LN97} or in Mullen and Panario \cite{handbook}.

For all missing definitions related to graphs, we refer the reader to Bollob\'{a}s \cite{bol98}.
Our primary object of study in this paper is defined as follows. Let $f\in \F[X,Y]$ and $g\in \F[X,Y,Z]$. Let $P$ and $L$ be two copies of the 3-dimensional vector space $\F^3$. We will refer to vertices of $P$ as {\it points}, and to those of $L$ as {\it lines}. To distinguish between vectors from $P$ and $L$, we will use parentheses and brackets, respectively.  Consider a bipartite graph $\Gamma _\F (f, g)$ with vertex partitions $P$ and $L$ and with edges defined as follows: for every $(p)=(p_1,p_2,p_3)\in P$ and every $[l]= [l_1,l_2,l_3]\in L$, $\{(p), [l]\} = (p)[l]$ is an edge in $\Gamma_3 = \Gamma _3 (f, g)= \Gamma _\F (f(p_1,l_1), g(p_1,p_2,l_1))$ if
$$
p_2+l_2 =f(p_1,l_1) \;\;\;\text{and}\;\;\;
     p_3 + l_3 = g(p_1,p_2,l_1).
$$
This graph has $2q^3$ vertices, and it is easy to see that a neighbor of any vertex is completely determined by the neighbor's first component. Hence, the degree of each vertex is $q$.

Graphs $\Gamma_\F (f,g)$ were studied recently due to a close relation of the graph $\Gamma_\F (p_1l_1, p_1l_1^2)$ to remarkable graphs called generalized quadrangles, whose definition we omit. For the details of this connection, generalizations of these graphs to other dimensions, and applications in finite geometries, extremal graph theory, and cryptography, see surveys by Lazebnik and Woldar \cite{lazwol2001}, a more recent one by Lazebnik, Sun and Wang \cite{lazsunwang2017}, and many references therein.  For recent results not discussed in \cite{lazsunwang2017},  see Nassau \cite{nassau2020};  Kodess, Kronenthal, Manzano-Ruiz and Noe \cite{kodkro2021};    Xu,  Cheng, and  Tang  \cite{xu2021}.

\medskip

The main question considered in this paper appeared in  Nassau \cite{nassau2020}.

\noindent{\bf Question 1 (\cite{nassau2020})}  {\it Given a graph $\Gamma_3= \Gamma _\F (f(p_1,l_1), g(p_1,p_2,l_1))$, is there a polynomial $h\in \F[X,Y]$, such that the graph  $\Gamma_2= \Gamma _\F (f(p_1,l_1), h(p_1,l_1))$:
$$p_2+l_2 =f(p_1,l_1) \;\;\;\text{and}\;\;\;
     p_3 + l_3 = h(p_1,l_1),$$
   is isomorphic to graph $\Gamma_3$ for infinitely many $q$?}

  Here the indices 2 and 3 in the notation for introduced  graphs reflect the number of variables present in functions $g$ and $h$.  There are many examples where the answer to Question 1 is affirmative. Moreover, given $g$, the polynomial $h$ can be found in many ways. For example, it is easy to verify that the following graphs are isomorphic:
    $$\Gamma _\F (p_1l_1, p_1^2l_1 - p_1p_2) \simeq \Gamma _\F (p_1l_1, p_2l_1)\simeq \Gamma _\F (p_1l_1, p_1l_1^2)\simeq \Gamma _\F (p_1l_1, p_1^2l_1).$$
    On the other hand,  it is not clear how to exhibit an infinite family of graphs
    $\Gamma_3$  for which the answer to Question 1 is {\it negative}.

  Another motivation for establishing that the family of all  graphs $\Gamma_2$  is a proper subset of the family of all  graphs $\Gamma_3$ (up to isomorphism) is related to the study of graphs with many edges and high girth.   The {\it girth}  of a graph containing cycles is the minimum length of all of its cycles. For odd $q\ge 3$, the only known graphs of the form  $\Gamma_3$ of girth $8$ are isomorphic to $\Gamma _\F (p_1l_1, p_1l_1^2)$, which is of the form $\Gamma_2$.   In particular, for infinitely many odd $q$ and several classes of polynomials $f$ and $h$, it has been proven that $\Gamma_2 (f(p_1,l_1),h(p_1,l_1))$ has girth 8 only when it is isomorphic to $\Gamma _\F (p_1l_1, p_1l_1^2)$ (see references preceding Question 1).  This suggests to focus the search for graphs $\Gamma_3$ with girth 8 to those that are not of the form $\Gamma_2$ (up to isomorphism).

Similarly to graphs $\Gamma_2$ or $\Gamma_3$,  we define a graph $\Gamma (f)= \Gamma _\F (f)$ in the following way.  Let $A$ and $B$ be two copies of $\F^2$. For $f\in \F [X,Y]$, consider a bipartite graph $\Gamma (f)= \Gamma _\F (f)$ with vertex partitions $A$ and $B$ and with edges defined as follows: for every $(a)=(a_1,a_2)\in A$ and every $[b]= [b_1,b_2]\in B$, $\{(a), [b]\} = (a)[b]$ is an edge in $\Gamma  (f)$ if
$$a_2+b_2 = f(a_1,b_1).$$
Consider a surjective $q$-to-1 map $\gamma$ of the vertex set of graph $\Gamma_3(f,g)$ to the vertex set of graph $\Gamma (f)$ defined by deleting the third component of each vertex. It is obvious that if $(p)[l]$ is an edge in $\Gamma_3(f,g)$, then $\gamma ((p))\gamma([l])$ is an edge of $\Gamma (f)$, i.e., $\gamma$ is a {\it graph homomorphism}. Moreover,  $\gamma$ is a covering homomorphism, meaning that here the image of the  neighborhood of every vertex $v$ of $\Gamma_3(f,g)$ is the neighborhood of $\gamma(v)$ in $\Gamma (f)$. In this case, we  will also say that $\Gamma_3(f,g)$ is a \textit{cover} of $\Gamma (f)$  or a {\it $q$-lift} of $\Gamma (f)$. Let $Aut(\cdot)$ denote the automorphism group  of a graph. It follows from our work that even when $|Aut(\Gamma (f))|$ is large, it is possible to find a $g$ such that $|Aut(\Gamma_3(f, g))|$ is much smaller than $|Aut(\Gamma (f))|$, and actually is as small as possible for graphs $\Gamma_3 (f, g)$.

In the next section we describe an infinite  family of graphs  that furnish such an example.

\section{Automorphisms of graphs $\Gamma_2$ and $\Gamma_3$ } \label{sec2}


We note that every graph of the form $\Gamma_2= \Gamma _\F (f(p_1,l_1), g(p_1 l_1))$ has the following automorphisms:   For all $a,b\in \F$,  consider the map $t_{a, b}$ on the vertex set of the graph $\Gamma_2$, such that  $t_{a, b} ((p_1,p_2,p_3))= (p_1,p_2+a,p_3+b)$ and  $t_{a, b} ([l_1,l_2,l_3])= [l_1,l_2-a,l_3-b]$.
We call these automorphisms {\it translations} of $\Gamma_2$. It is obvious that all $q^2$ of $t_{a, b}$ form a subgroup of $Aut(\Gamma_2)$ that is isomorphic to the additive group $(\F^2, +)$ of the vector space $\F^2$. Therefore, the order of $Aut(\Gamma _2)$ must be divisible by $q^2$.

We note that every graph of the form $\Gamma_3= \Gamma _\F (f(p_1,l_1), g(p_1,p_2, l_1))$ has the following  automorphisms: For each $b\in \F$,   consider a map $t_b$ on the vertex set of $\Gamma_3$, such that  $t_b ((p_1,p_2,p_3))= (p_1,p_2,p_3+b)$ and  $t_b ([l_1,l_2,l_3])= [l_1,l_2,l_3-b]$. We call these automorphisms {\it translations} of $\Gamma_3$. It is obvious that all $q$ of $t_b$ form a group that is isomorphic to the additive group $(\F, +)$ of the field  $\F$.  Hence,  the order of
$Aut(\Gamma _2)$ must be divisible by $q$.

  Suppose $q=p$. If we can exhibit a graph $\Gamma_3 $ such that $|Aut(\Gamma_3)|$ is not divisible by $p^2$, it will provide a negative answer to
  Question 1.
 Several of such graphs were found by computer, see  \cite{nassau2020}. One of them was simplified to the form   $$R = \Gamma _{\F_p} (p_1l_1, p_1p_2l_1(p_1+p_2+p_1p_2)).$$  It was checked by computer that $|Aut(R)| = p$ for all prime $p$, $3\le p\le 41$, but the proof that it holds for all odd $p$ was based on several assumptions which are proven in this paper.  Our notation $R$ for the graph relates to the adjective ``rigid", and it reflects the fact that $R$ has the  smallest possible automorphism group for graphs of type $\Gamma_3$. It was shown in \cite{nassau2020}, that the diameter of $R$ is $6$ for all odd $q$,  $q>3$,  and for $q=3$, the diameter is $7$.
  It is also easy to check that $R$ contains no 4-cycles.

The main result of this paper is the following.
\begin{thm}\label{main}
For $p \equiv 1 \pmod{3}$,  and $R = \Gamma _{\F_p} (p_1l_1, p_1p_2l_1(p_1+p_2+p_1p_2))$. Then $|Aut(R)| = p$, and $R$ cannot be presented in  the form  $\Gamma _{\F_p} (p_1l_1, h(p_1,l_1))$.
\end{thm}

It is known that $|Aut(\Gamma_{\F_p}(p_1\ell_1))| = 2p^3(p-1)^2$, two different proofs can be found in Viglione \cite{Vig2002}. As $R$ is a $p$-lift of $\Gamma_{\F_p}(p_1\ell_1)$, by Theorem 2.1, it provides an example of our comment at the end of the previous section.

The idea of our proof of Theorem 2.1 is based on exhibiting two sets of vertices fixed by any automorphism of $R$ (see Section \ref{sec3}),  and then using these sets to conclude that  $Aut(R)\simeq (\F_p,+)$ (see Section \ref{sec4}).  We collect final remarks and state several open questions in Section~\ref{sec5}. In the appendix we present code of simple programs mentioned in the text.

 \section{Special subsets of lines in $R$}\label{sec3}

   In this section we consider  graphs $R = \Gamma _{\F_q} (p_1l_1, p_1p_2l_1(p_1+p_2+p_1p_2))$,  where $q$ is a power of an odd prime and $q\ge 7$.

For $i\ge 1$, let  $R^i(v)$ denote the set of vertices of $R$ at distance $i$ from a vertex $v$.  $R^i(v)$ is  often referred to as the {\it $i$-neighborhood} of  $v$ in $R$.  Let $r_i (v)= | R^i(v) |$.  It is clear that for every vertex $v$ of $R$, $r_1(v)=q$, and, as  $R$ contains no 4-cycles,    $r_2(v)=q(q-1)$.  It is easy to check that the value of $r_3(v)$ is not the same for all vertices.  The question we wish to ask is:  For which lines $[l]= [l_1,l_2,l_3]$ of $R$  $r_3([l])$ is maximum? As  the size of a 3-neighborhood of a vertex of $R$ is preserved by any automorphisms of a graph, answering this question  will be helpful to us.

Due to the translation automorphisms on the third components of vertices of $R$, it is sufficient to consider lines $[A,B,0]$ only.
\medskip

\begin{prop}(\cite{nassau2020}) \label{2maxnum} For $q\ge 7$, in graph $R$ $$r_3([0,1,0]) =   q^3 - 4q^2 + 9q -8\;\;\text{and } \;\; r_3([0,0,0]) = q^3 -4q^2 +8q -6. $$
\end{prop}
Our interest in $r_3([0,1,0])$  and  $r_3([0,0,0])$ is due to the fact that they  are the maximum and the second maximum, respectively, of the sizes of 3-neighbourhoods of lines in $R$. For completeness of the exposition, a rephrased  proof of  Proposition \ref{2maxnum} appears as a part of our proof of Theorem \ref{0001AB} below,  and it will exhibit the notions needed for the proof of the theorem.
\begin{thm}\label{0001AB} Let $q\ge 7$,  and $q\equiv 1 \pmod{3}$. Then

(i) For any line $[A,B, C]$ of $R$ such that   $(A,B) \neq (0,0), (0,1)$, $$ r_3([A,B,C]) < r_3([0,0,0]) < r_3([0,1,0]).$$

(ii) For any point  $(p)= (p_1,p_2,p_3)$ of $R$,   $$ r_3((p_1,p_2,p_3)) < r_3([0,0, 0]).$$
\end{thm}

\begin{proof}
For all odd prime powers $q$, $7 \leq q \leq 41$, the statement has been verified by computer.
Therefore in all our arguments, we can assume $q \geq 43$, even if a smaller lower bound of $q$  suffices.

We write $x\sim y$, if  $\{x,y\}$ is an edge of $R$.   Let $$[A,B,0] \sim (a,*,*) \sim [x, *,*]\sim (b, A a-a x+b x-B, *)$$  be a path of length 3 starting at a line $[A,B,0]$. Here $*$ represents expressions whose explicit form is not important to us. Hence,  $A\neq x$ and $a\ne b$, as otherwise the corresponding vertices coincide. In what follows we use a Maple program to facilitate straightforward symbolic computations, the code can be found in the appendix. Let  $c = A a-a x+b x-B$. Then  $x= (c-Aa +B)/(b-a)$,
and the condition $x\neq A$ is equivalent to ($c\neq Ab-B$ and   $a\neq b$).  Substituting this expression for $x$ to $(b, A a-a x+b x-B, *)$,  we obtain the following representation of $R^3([A,B,0])$:
\begin{equation}\label{3AB0}
R^3([A,B,0]) = \Big\{\Big( b,c, \frac{P_{A,B}(b,c;a)}{b-a}\Big):  \;\; a,b,c,\in \F, c\neq Ab-B, a\neq b \Big\},
\end{equation}
where $P_{A,B}(b,c;a)$,  viewed as a polynomial of $a$, is
\begin{multline}\label{P3AB}
P_{A,B}(b,c;a) =
A^2(Ab-B-c)a^4+A(A-2B+1)(Ab-B-c)a^3-B(2A-B+1)(Ab-B-c)a^2+ \\
(-Ac^2b^2+AB^2b-Ac^2b-ACb^2-B^3-B^2c)a+cb(cb+c+b)(c+B).
\end{multline}

Specializing $(A,B)= (0,0)$  and $(A,B)= (0,1)$,  we obtain
\begin{equation} \label{300}
R^3([0,0,0]) = \Big\{\Big( b,c, \frac{bc^2(bc + b+c)}{b-a}\Big):  \;\; a,b,c,\in \F,  c\neq 0, a\neq b \Big\}
\end{equation}
and
\begin{equation} \label{301}
R^3([0,1,0]) = \Big\{\Big( b,c, \frac{(bc(bc+c+b)-b)(c+1)}{b-a} + (c+1)\Big):  \;\; a,b,c,\in \F, b\neq a, c\neq -1 \Big\}.
\end{equation}
The representations  (\ref{3AB0}) and (\ref{300}) allow  us to compare $r_3([0,0,0])$ and  $r_3([A,B,0])$ by going over all  $q(q-1)$ choices of $(b,c)$,   and,  for each fixed $(b,c)$, by  comparing the sizes of ranges of the functions defining the 3rd components of the verices from the corresponding 3-neighborhoods, namely   $$ \frac{bc^2(bc + b+c)}{b-a}\;\;\; \text{and}\;\;\; \frac{(bc(bc+c+b)-b)(c+1)}{b-a} + (c+1),$$  where these functions are viewed as functions of $a$ on $\F\setminus \{b\}$.
\medskip

$\;\;\;$   \underline{Case 1 (\textit{i}):  Suppose $A=B=0$.}  We will use equality (\ref{300}).  If $bc^2(bc + b+c)\neq 0$,  then $bc^2(bc + b+c)/(b-a)$ takes $q-1$  distinct values,  one for each $a$ distinct from $b$, i.e., the  greatest possible number of values for any function of $a$ with domain of size $q-1$. On the contrary,  when $bc^2(bc + b+c)= 0$, then  $0/(b-a) = 0$ for each $a$ distinct from $b$. This corresponds to the  least possible number of values, namely 1,  that any (non-empty) function of $a$ can have.

Note that $bc^2(bc + b+c)= 0$ and $c\neq 0$ imply that $b\neq -1$. Therefore all solutions of the system    $(c\ne 0 \;\text{and}\; bc^2(bc + b+c)= 0)$ can be describe as
$$Z = \{(0,c): c\neq 0\}\; \bigcup \;\{(b, -b/(b+1)): b\neq 0,  b\neq -1\}.$$ Hence, $|Z|= q-1 + q-2 = 2q-3$.   Therefore, the part of  $R^3([0,0,0])$ formed by points $(b,c,0)$ with $(b,c)\in Z$ contains $2q-3$ points,  and the remaining part, consisting of points of the form $(b,c, bc^2(bc + b+c)/(b-a)$, with  $(b,c)\not\in Z$ contains
$$[q(q-1) - (2q-3)](q-1) = (q^3-3q^2 +3)(q-1)= q^3-4q^2 + 6q -3$$ points. This proves that
$r_3([0,0,0]) = (q^3-4q^2 + 6q -3) + (2q-3) = q^3-4q^2+8q - 6$, as stated in  Proposition \ref{2maxnum}.
\medskip

$\;\;\;$ \underline{Case 2 (\textit{i}):  Suppose $A=0, B=1$.}

We will use equality (\ref{301}).  Note that the system   $((bc(bc + b+c)- b)(c+1)= 0 \;\text{and}\;c\neq -1)$
is equivalent to $(bc(bc + b+c) -b= 0 \;\text{and} \; c\neq -1)$,
which is equivalent to
$$(b=0 \;\text{and} \; c\neq -1 ) \; \text{or} \; [(b\neq 0 \;\text{and}\; c(bc + b +c) -1 =0)].$$
The first system gives  $q-1$ points of the form $(0, c, c+1)$.
The second system is equivalent to  $ (b(c+1)c =(1-c)(1+c) \;\text{and}\; b\neq 0 \;\text{and}\; c\neq -1)$, or to just $( bc =(1-c) \;\text{and}\; b\neq 0 \;\text{and}\; c\neq -1)$. For $c=1$,  it has no solution. For $c\neq 1, -1$, it has a unique solution $b=(1-c)/c$.
Therefore the second system gives  $q-2$ points of the form $((1-c)/c, c, c+1)$, $c\neq 0,-1)$.  Hence, there are $(q-1) + (q-2) = 2q-3$  pairs $(b,c)$ such that each contributes 1 point to $R^3([0,1,0])$.  They are of the form  $(b, c, c+1)$.
Of the  remaining possible pair $(b,c)$, $c\neq -1$, each contributes $q-1$ points of the form $(b,c, (bc(bc+c+b)-b)(c+1)/(b-a) + (c+1))$, when $a$ runs over all elements of $\F$ except $b$. As there are $q(q-1) - (2q-3)= q^2-3q+3$ of those pairs, they all  contribute $(q^2-3q+3)(q-1)$ points. Therefore,
$r_3([0,1,0]) = (q^2-3q+3)(q-1) + (2q-3) = q^4 - 4q^2 +9q -8$. This proves the remaining case of Proposition \ref{2maxnum}.

 \medskip

$\;\;\;$ \underline{Case 3 (\textit{i}):  Suppose $A=0, B\ne 0,1$.}

Using formula (\ref{P3AB}),  for $A=0$,  we obtain:
$$P_{0,B}(b,c;a) =  -B(B+c)[(B-1)a^2 + B a]+ cb(cb+c+b)(B+c). $$
Setting $t=b-a$, it is each to check we can rewrite $P_{0, B}(b, c; a)/(b-a)$ in the form
$$ \frac{P_{0,B}(b,c;a)}{b-a} =  \alpha_1 t + \alpha_0  + \frac{\alpha_{-1}}{t} =: f(t),  $$ where $\alpha_1 =  (B-1)B(B+c)$, $\alpha_0 = B(B+c)(2(B-1)b + B$, and $\alpha_{-1} = b(B+c)^2(Bb-bc+B-b-c).$
Note that the condition $c\neq 0\cdot b -B$, is equivalent to $B+c\neq  0$.  When $a$ changes over $\F\setminus \{b\}$, $t$ changes over $\F^\times $,  and for fixed $B\neq 0,1$, $b$ and $c$, $B+c\neq 0$,  the function $f$ takes as many values as the function  $g$, which is obtained from $f$ by dropping $\alpha_0$ and dividing by $\alpha_1 \neq 0$: $$g(t) := t + \frac{b(Bb - bc +B -b-c)}{t} = t + \frac{\gamma}{t},  \; t\ne 0,$$
where $\gamma = b(Bb - bc +B -b-c)$. For $\gamma = 0$,  $g$ takes $q-1$ values. This happens when
$b=0$, or $b= -(B^2-c^2)/(2B^2-c^2-2B-c)$, provided $2B^2-c^2-2B-c\neq 0$. Hence, $\gamma=0$ for at most $ 2q-1 $ pairs $(b,c)$, which are of the form  $\{(0,c),  c\neq -B\}$, or
$\{(-(B^2-c^2)/(2B^2-c^2-2B-c), c), c\neq -B, 2B^2-c^2-2B-c\ne 0.\}$. For  other pairs $(b,c)$, i.e., for at least $q(q-1) - (2q-1) = q^2-3q+1$ pairs,  $\gamma \neq 0$,  and  $g$ takes equal values for every  $t_1\neq 0$ and  $t_2 = \gamma /t_1$. Hence, for these $(b,c)$, $g$ takes at most $(q-1)/2$ distinct values.
Therefore, $$r_3 ([0,B,0]) \leq  (2q-1)(q-1) + (q^2-3q+1)(q-1)/2 <   q^3-4q^2+8q - 6 = r_3([0,0,0]),$$
for all $q\ge 5$.  This proved the second inequality of part (i) of Theorem \ref{0001AB} in this case.

 \medskip

$\;\;\;$ \underline{Case 4 (\textit{i}):  Suppose $A\neq 0$.}

We will show that in this case, for fixed $A,B,b,c\in \F_q$,  $A\neq 0$, $c\neq Ab-B$, the expression $P_{A,B} (b,c; a)/(b-a)$,  considered as a rational function of $a$ on $\F_q\setminus \{b\}$,  takes at most $q-3$ distinct values.
 Then, by (\ref{3AB}),
$$r_3 ([A,B,C]) \leq q(q-1)(q-3) = q^3 -4q^2 + 3q < q^3 -4q^2 +8q -6 = \gamma_3([0,0,0]),  $$
and this will end the proof of part (i) of  Theorem  \ref{0001AB}.
\bigskip

Using (\ref{P3AB}),  and $x= b-a$, it is easy to check that we can rewrite $P_{A,B}(b,c;a)/(b-a)$ in the form
$$ \frac{P_{A,B}(b,c;a)}{b-a} = \alpha_3 x^3 + \alpha_2 x^2 + \alpha_1 x + \alpha_0 + \frac{\alpha_{-1}}{x} =: h(x),$$
where $\alpha_3 = A^2(Ab-B-c) \neq 0$, as $A\neq 0$ and $c\neq Ab-B$.
The explicit expressions for $\alpha_i$, $i=2,1,0,-1$,
as functions of $A,B, b,c$ will not matter to us.

It is clear that the function $h$ takes as many distinct values on  $ \F^\times $ as the function $j$, obtained from $h$ by dropping the constant term $\alpha_0$, and dividing all coefficients by the nonzero coefficient $\alpha_3$:
 $$j(x) := x^3+ c_2 x^2 + c_1 x  + \frac{c_{-1}}{x},$$
where, again,   the explicit expressions for $c_i$, $i=2,1,-1$, will not matter to us.

As for each  $t\in \F^\times$, $1/t= t^{q-2}$,  the rational function $j:   \F ^\times \to \F$,
can be represented by  a polynomial function $j(x)= x^3+ c_2 x^2 + c_1x  + c_{-1}x^{q-2}$.  The same polynomial $J=j= X^3+ c_2 X^2 + c_1X  + c_{-1}X^{q-2}$ can be used to define a function $J$ on $\F$, by assuming that $J(0)=0$ and $J(t)=j(t)$ for all $t\in \F^\times$.
Hence, function $j$ is a restriction of function $J$ to $\F^\times$.  For any function $f:\F\to \F$, let $V_f$ denote the range of $f$,  and $\#V_f = |V_f|$.  Then   $\#V_j \le \#V_J$, and the inequality is strict if only only if $0\not\in V_j$.

Our goal now is to prove that under certain conditions on $q$, $\#V_j \le q-3$. As was remarked at the beginning of Case 4,  this will suffices to finish the proof of part (i) of Theorem \ref{0001AB}.

If  $f:\F\to \F$ is a bijection, then $V_f = \F$, $\#V_f = q$,  and $f$ is called a {\it permutation polynomial} (or PP)  in $\F$.
\begin{lem}\label{lem} Let $q\ge 17$,  $q\equiv 1 \pmod{3}$,  and
 $J = X^3+ c_2 X^2 + c_1X  + c_{-1}X^{q-2} \in \F[X]$.
 Then the following holds.
 \begin{enumerate}
 \item  $J$ is not a PP.\label{lem1}
 \item $\#V_j \le \#V_J \le q-2$.  If, in addition,  $0\not\in V_j$, then $\#V_j \le q-3$. \label{lem2}
 \end{enumerate}
 \end{lem}
 \begin{proof}  1.   We will need the following  fact, often  referred to as Hermite-Dickson Criterium (see Hermite \cite{hermite1863}, Dickson \cite{dickson1897} or  \cite{LN97}, \cite{handbook}).
\begin{prop} {\rm (Hermite-Dickson)}\label{Hermite} Let $p$ be the characteristic of $\F=\F_q$. A polynomial $f \in \F[X]$ is PP if and only if the following two conditions hold: (i) $f$ has exactly one root in $\F$, and, (ii) for each integer $t$ with $1\le t\le q-2$ and $t\not\equiv 0 \pmod{p}$,  the reduction of $(f(X))^t$  modulo $(X^q-X)$ has degree at most $q-2$.
	\end{prop}
 Let us  show that the condition (ii) of   Proposition \ref{Hermite} fails for our polynomial $J$.  We can equate the the coefficients at $X^{q-1}$ in $J^n$  modulo $(X^q-X)$, $n=2,3,4$,   to zero. This leads to the following system of equations:
\begin{eqnarray}
2 c_1c_{-1}&=& 0 \;\; \;\text {(using $J^2$)} \label{Jsq}\\
 3c_2c_{-1}^2  &=& 0 \;\; \; \text {(using $J^3$)}\label{Jcube}\\
4c_{-1}^3 +  6c_1^2c_{-1}^2   &=& 0 \;\; \; \text {(using $J^4$ )}\label{Jfour}
\end{eqnarray}
We wish to remark that for $q<17$, the  the coefficient at $X^{q-1}$ in $J^4$  modulo $(X^q-X)$ takes different forms than in (\ref{Jfour}).  This explains the condition $q\ge 17$.
%

For $p\ge 5$, if $c_{-1} \neq 0$,  then (\ref{Jsq}) gives $c_1=0$ and  (\ref{Jfour}) gives  $c_{-1}=0$, a contradiction.
Therefore $c_{-1} = 0$
  and $J=c_1X + c_2X^2 + X^3$ is a monic cubic polynomial.  It is well known, see Dickson \cite{dickson1897} or \cite{LN97}, \cite{handbook},  that for  $q\equiv 1 \pmod{3}$,
  a cubic polynomial  cannot be a PP. (Moreover,  if $d>1$ is a  divisor of  $q-1$,
  $q$ being  any prime power, then there is no PP of degree $d$.) Hence,  $J$ is not a PP.
 \bigskip

 {\bf Remark.}  If $q \equiv 0,2 \pmod{3}$,  there exist cubic PP's.   In this case Theorem \ref{0001AB} still holds, but the proof becomes more subtle. As it does not affect our main Theorem \ref{main}, we decide not to pursue it here.
\bigskip

   2. If $J$ is not a PP, then $\#V_J \le q-1$.  The following result allows us to decrease the upper bound by 1.
\begin{prop}{\rm (Wan \cite{Wan})} \label{Wan}  If a polynomial $f$ of degree $n\ge 1$ is not a PP of $\F$,  then $$\#V_f \leq q - \Big\lceil \frac{q-1}{n} \Big\rceil.$$
\end{prop}
As $J$ has degree $q-2$ or 3,  and is not a PP of $\F$,
Proposition  \ref{Wan}, implies
$$\#V_J \leq q - \Big\lceil \frac{q-1}{q-2} \Big\rceil = q-2. $$
As $\#V_j \le \#V_J$, if $0\not\in V_j$, then $\#V_j =\#V_J -1 \le q-3$. This ends the proof of Lemma~\ref{lem}.
\end{proof}

What left is to analyze the case $\#V_j= q-2$,  as for   $\#V_j\leq  q-3$ the theorem has been proven. So  we assume that $\#V_j= q-2$. Then $0\in V_j = V_J$, and
 function $j$, having  domain  $\F^\times$, takes some value $c$ exactly twice,  and each other value  exactly ones. Let $j(x_1)=j(x_2)=c$, $x_1\neq x_2$. Then $\F \setminus V_j = \{a,b\}$, where $a, b$ are  distinct, and none of $a$ or $b$ is equal to $c$ or to 0.   Consider the following three polynomial functions on $\F$:

%

$$J_{a-c}(x) = \frac{\prod_{t\in \F^\times,\, t\neq x_1} (x-t)}{\prod_{t\in \F^\times,\, t\neq x_1} (x_1-t)}\,(a-c) + J(x) -c,\;\;\;\;\;
J_{b-c}(x) = \frac{\prod_{t\in \F^\times,\, t\neq x_2} (x-t)}{\prod_{t\in \F^\times,\, t\neq x_2} (x_2-t)}\,(b-c) + J(x)-c,$$
$$J_{a-c,b-c}(x) = \frac{\prod_{t\in \F^\times,\, t\neq x_1} (x-t)}{\prod_{t\in \F^\times,\, t\neq x_1} (x_1-t)}\,(a-c) +  \frac{\prod_{t\in \F^\times,\, t\neq x_2} (x-t)}{\prod_{t\in \F^\times,\, t\neq x_2} (x_1-t)}\,(b-c) + J(x)-c.$$
 As the product of all elements of $\F^\times$ is $-1$, we have:
 $$J_{a-c}(x_1) = a-c, \; J_{a-c}(x_2) = 0, \; J_{a-c}(0) = (c-a)/x_1 -c,$$
 $$J_{b-c}(x_1) = 0, \; J_{b-c}(x_2) = b-c, \; J_{b-c}(0) = (c-b)/x_2 -c,$$
 $$J_{a-c,b-c}(x_1) = (c-a)/x_1, \; J_{a-c,b-c}(x_2) = (c-b)/x_2, \; J_{a-c,b-c}(0) = (c-a)/x_1 + (c-b)/x_2 -c.$$
 The degrees of polynomials $J_{a-c}, J_{b-c}$, and $J_{a-c,b-c}$ are at most $q-2$.  If one of them takes exactly $q-1$ values, then it is not a PP, and we get a contradiction with Theorem \ref{Wan}.   Therefore each of these polynomials must take  either $q-2$ or $q$ values.

For $t\not\in \{0,x_1\}$, $J_{a-c}(t)= J(t) -c =j(t)-c$. This gives already $q-2$ distinct values of $J_{a-c}$, each different from $a-c$. Hence, $\#V_{J_{a-c}} \geq q-1$. So we must have $\#V_{J_{a-c}} = q$. This is possible if and only if $J_{a-c}(0)= (c-a)/x_1 -c =  b-c$, or equivalently, $(c-a)/x_1 = b$.

Similarly we conclude that $\#V_{J_{b-c}} = \#V_{J_{a-c,b-c}}= q$, which happens if and only if $J_{b-c}(0)= (c-b)/x_2  = a$, and  $J_{a-c,b-c} (0) = (c-a)/x_1 + (c-b)/x_2 -c = 0$. Therefore, the following three equalities must be satisfied  simultaneously: $$(c-a)/x_1  = b, \;\;\;  (c-b)/x_2 = a,\;\;\; (c -a)/x_1 + (c-b)/x_2 - c= 0.$$
From the first and the third  equalities, we obtain $(c-b)/x_2 = c-b$, and from the second and the third equalities, we get $(c-a)/x_1 = c-a$.  As $c$ is distinct from $a$ and from $b$,  we obtain  $x_1=x_2=1$, which contradicts the assumption that $x_1\neq x_2$. Hence, $\#V_j \le q-3$.   This ends the proof of part (i) of Theorem \ref{0001AB}.

%
\bigskip

Our proof of part (ii) of  Theorem \ref{0001AB} is very similar to the one of part (i),  and so we will proceed through  it a bit faster.  In order to prove part (ii), due to the translation automorphisms on the third coordinates of vertices of $R$, it is sufficient to consider points $(A,B,0)$ only.   Let $$(A,B,0) \sim [x,*,*] \sim (a, *,*)\sim [y, Ax-ax+ay-B, *]$$  be a path of length 3 from the the point $(A,B,0)$. Hence,  $A\neq a$ and $x\ne y$. Let  $z = Ax-ax+ay-B$. Then  $a= (Ax-B-z))/(x-y)$,
and the condition $A\neq a$ is equivalent to ($z\neq Ay-B$ and   $x\neq y$).  Substituting this expression for $a$ to $[y, Ax-ax+ay-B,*]$, and then setting $x=y+1/t$, $t\neq 0$,  we obtain the following representation of $R^3((A,B,0))$:
\begin{equation}\label{3AB}
R^3((A,B,0)) = \Big\{\Big[ y,z, Q_{A,B}(y,z;t)\Big]:  \;\; y,z,t\in \F, z\neq Ay-B, t\neq 0 \Big\},
\end{equation}
where $Q_{A,B}(y,z;t)$, viewed as a function  of $t$  with the constant addend dropped, is of the form $j(t) = c_3t^3 + c_2t^2+c_1t + c_{-1}/t$, with $$c_3= -y^2(Ay-B-z)^4 \;\;\;\text{and}\;\;\; c_{-1} = A(Ay-B-z)((A^2+A)y+(B-z+1)A+B-z).$$

As $z\neq Ay-B$, we continue our analysis by considering  cases: $A= 0$ and $A\neq 0$.

\bigskip


$\;\;\;$ \underline{Case 5 (\textit{ii}): $ A=0$}


$\;\;\;$   For \underline{$y=0$} and $z\neq B$, we arrive to considering   $\# V_{j}$ for $j(t) = z(1-z)t$.
For $z(1-z) \neq 0$ and  $z\neq B$, $j$ takes at most $q-1$ values.  As there are at most $q-3$ such values of $z$, there  at most $q-3$ pais $(0,z)$ for which $j$ takes $q-1$ values.
Their total contribution to $r_3((0,B,0))$ is at most  $(q-1)(q-3)$.

If $z(1-z)=0$, and $z\ne B$, then $z$ can take at most 2  values,  and so each of the corresponding pairs $(0,z)$ contributes at most 1 to $r_3((0,B,0))$.

Therefore, the contribution of pairs $(0,z)$  to $r_3((0,B,0))$ is at most
$$(q-1)(q-3) + 2 = q^2-4q+5.$$

Suppose now  that \underline{$y\neq 0$}. Then we need to estimate $\# V_{j}$, where $j$ is a monic cubic polynomial of $t\in \F^\times$. As  it is not a PP ($q\equiv 1 \pmod{3}$), $\# V_{j}\leq (2q+1)/3$.
As there are at most $q(q-1)$ pairs $(y,z)$, their total contribution in $r_3((0,B,0))$ is at most  at most $q(q+1)(2q+1)/3$,

Therefore,  $$r_3((0,B,0)) \le  q^2-4q+5 +  q(q+1)(2q+1)/3 < q^3-4q^2+8q-6 = r_3([0,0,0]),$$
for any $q\ge 5$,  and part (ii)  of Theorem \ref{0001AB} is proven in this case.
\bigskip

$\;\;\;$ \underline{Case 6 (\textit{ii}): $A\neq 0$}

$\;\;\;$  If \underline{$y=0$},  we arrive to the investigation of $\#V_j$ for $$j(t) = -z(z-1)t + (AB+A+B - (A+1)z)/t.$$ For $z(z-1)\neq 0$, it is reduced to investigating the range of the function of the form $ j(t) = t + \gamma /t$, with $\gamma = (AB+A+B - (A+1)z)/((B+z)z(1-z))$. In this case we proceed as in   Case 3, noting that for   $\gamma \neq 0$, the values of $j$ at $t$ and at $\gamma/t$ are equal. There are at most $q-2$ values of $z$ such that $z+B\neq 0$ and $\gamma\neq 0$. For each of them there are at most $(q-1)/2$ values of $j$. Now, $\gamma = 0$ for at most 1 value of $z$, and for this $z$, $j$ takes at most $q-1$ values.  Hence,  the contribution of all such pairs $(0,z)$ to $r_3(( A,B,0))$  is at most  $(q-2)(q-1)/2 + (q-1)$.

If $z(z-1)=0$, then  $z=0$ (and so $B\neq 0$) or $z=1$ (and so $ B+1\neq 0$). In these cases,  we obtain  that $j(t)= (AB+A+B)/t$  or $(AB+B-1)/t$,  and so it takes either 1 or at most $q-1$ values, depending on the numerators being 0 or not. This contributes to at most  $2+ 2(q-1)$ lines  in $R^3((A,B,0))$.

$\;\;\;$ If \underline{$y\neq 0$}, then dividing $j(t)$ by $c_3$ and dropping the constant term, leads to estimating $\# V_j$ of  the form $j(t) = t^3 + c_2t^2+c_1t + c_{-1}/t$, $t\neq 0,$ with
 $$ c_{-1} = A(A^2y+AB+Ay-Az+A+B-z)/((Ay-B-z)^3y^2).$$
If   $c_{-1} = 0$, then $j$ is a monic cubic polynomial of $t$. If $c_{-1} \neq 0$, we proceed  
as we did in  Case 4, and conclude that $\# V_j \leq q-3$. Hence, the contribution of at most $(q-1)^2$ pairs $(y,z)$ into $r_3(( A,B,0))$ is at most $(q-1)^2(q-3)$.

Combining all our findings, we obtain
$$r_3((A,B,0)) \leq  (q-2)(q-1)/2 + (q-1) +   2+ 2(q-1) + (q-1)^2(q-3)  < $$
$$ q^3 -4q^2 +8q -6 = r_3([0,0,0]),$$
for all $q\ge 3$. This ends the proof of  part (ii) of Theorem \ref{0001AB}, and so of the theorem.  \end{proof}

An  immediate corollary of part (ii) of  Theorem \ref{0001AB} is that a line (point) of $R$ cannot be mapped to a point (line) of $R$ by any automorphism of $R$. It will be used in the next section.
\begin{cor} Let $q \equiv 1 \pmod{3}$.  For every  automorphism $\phi$ of $R$, $\phi (P) = P$ and $\phi (L) = L$.
\end{cor}

\section{Proof of Theorem \ref{main}} \label{sec4}

    Though Theorem 2.1 is stated for $q$ prime, the requirement for $q$ to be prime is only utilized at the end for the conditions of Corollary 4.11 to be met. At the same time, many of the statements we prove to establish Theorem 2.1 are true for odd prime powers $q \equiv 1$(mod 3). Therefore, in all of the following statements we assume that $q$ satisfies these conditions without repeating it each time and $R$ is defined over $\F_q$.
In the following lemma, we collect some observations about the distance between specified vertices.

\begin{lem}\label{Dist6verts}
The following holds in $R$:
\begin{enumerate}[(i)]
    \item Any two vertices of $R$ at distance 2, have distinct first components.

    \item For $b \neq c$, the points $(0, b, r)$ and $(0, c, s)$ in $R$ are at distance 4.

    \item For $r \neq s$, the lines $[x, y, r]$ and $[x, y, s]$(or points $(x, y, r)$ and $(x, y, s)$) are at distance at least 6.

\end{enumerate}

\end{lem}

\begin{proof} (\textit{i}) If two vertices are at distance 2, then they are distinct neighbors of another vertex. If their first components are equal, then the definition of the adjacency in $R$ implies that the second and third components are equal, a contradiction.

(\textit{ii}) By part (\textit{i}), the points $(0, b , r)$ and $(0, c, s)$,  where $b \neq c$, cannot be at distance 2. Choose $m$ such that $m \neq 0, -1$ and $m^2 \neq (s-r)/(c-b)$. As $q \geq 5$, such an $m$ is guaranteed to exist. Then it is easy to verify that the following adjacencies define a path of length 4 between the given points:

$$
(0, b, r) \sim [(m-b)/z, -b, -r] \sim  (z, m, (m-b)m((m+1)z+m) + r) =
$$
$$
(z, m, (m-c)m((m+1)z+ m) +s) \sim [(m - c)/z, -c, -s] \sim (0, c, s),
$$
where
$$
z = \left(\frac{s-r}{b-c} - m^2\right)/(m^2 + m)
$$

makes the third component of the middle points equal and guarantees $z \neq 0$.

(\textit{iii}) We will show that these two lines (two points) are not at distance two or distance four from each other. Clearly, they are not at distance 2 by part (\textit{i}).

Suppose the two lines are at distance 4. Then
$$
[x, y, r] \sim (z, zx - y, *) \sim [\alpha, \alpha z - zx +y, * ] = [\alpha, \alpha w - wx +y, * ]\sim (w, wx - y, *) \sim [x, y, s],
$$
with $w \neq z$ and $x \neq \alpha$. Then $\alpha z - zx + y = \alpha w - wx +y$, which is equivalent to
$
(z-w)(\alpha - x) = 0,
$
a contradiction. Thus $[x, y, r]$ and $[x, y, s]$ are at distance at least 6. Since $f_2(p_1, \ell_1) = p_1\ell_1$ is symmetric with respect to $p_1$ and $\ell_1$, then the same exact argument works for points of the same form. This ends the proof. We wish to comment again that in \cite{nassau2020} the diameter of $R$ was shown to be 6 for odd $q \geq 5$.
\end{proof}

Denote $\textit{Aut}(R)$ by $G$, and let $\phi \in G$. To simplify notation, instead of $\phi([x, y, z])$ we will just write $\phi[x, y, z]$, similarly $\phi((x, y, z)) = \phi(x, y, z)$. By Corollary 3.6, $G$ acts on the set of lines and the set of points of $R$, so then we may write,
\begin{align*}
    \phi[x, y, z] = [\lambda_1(x, y, z), \lambda_2(x, y, z), \lambda_3(x, y, z)], \\
    \phi(x, y, z) = (\pi_1(x, y, z), \pi_2(x, y, z), \pi_3(x, y, z)),
\end{align*}
where $\lambda_i$ and $\pi_i$ are component functions of $\phi$. The notation $\lambda_i$ and $\pi_i$ will remind us that they correspond to the action of $\phi$ on lines and on points. From here on, we will assume that $\lambda_i$ and $\pi_i$ implicitly depend on $\phi$.

For a fixed $a \in \F$, the following sets will play an important role in our arguments:
$$
L_{a} = \{ [0, a, r]: r \in \F \} \hs \text{and}  \hs P_{a} = \{ (0, a, r): r \in \F \}.$$

An immediate corollary of Theorem \ref{0001AB} is as follows.
\begin{lem}\label{FelThm}
$G$ acts on $L_0$ and $L_1$.
\end{lem}



The goal of the following statements is to ultimately prove that the action of any $\phi \in G$ on any component is determined only by that component. Meaning, that $\lambda_i$ and $\pi_i$ for $1 \leq i \leq 3$, can be reduced to a single variable function. For these new single variable functions, we will use the same notation, that is $\lambda_i(v_1, v_2, v_3) = \lambda_i(v_i)$ and $\pi_i(v_1, v_2, v_3) = \pi_i(v_i)$.

\begin{lem}\label{GactsP0}
$G$ acts on $P_0$.
\end{lem}

 \begin{proof}
Let $r \in \F$. For each $x \in \F^\times$, consider the path
\begin{equation}\label{Eq13}
    [0, 1, r] \sim (x, -1, -r) \sim [-x^{-1}, 0, r-1] \sim (0, 0, -r+1) \sim [0, 0, r-1].
\end{equation}

Let $\phi$ be in $G$. By Lemma \ref{FelThm}, we have that $\phi[0, 1, r] = [0, 1, w]$ and $\phi[0, 0, r-1] = [0, 0, s]$ for some $w, s \in \F$. Then applying $\phi$ to (\ref{Eq13}) yields
\begin{equation}\label{Eq14}
    [0, 1, w] \sim (x', -1, -w) \sim [y', x'y'+1, x'y'+w] = [y', zy', s] \sim (z, 0, -s) \sim [0, 0, s],
\end{equation}
where $y' \in \F^\times$.

As $\phi(0, 0, -r+1) = (z, 0, -s)$, then $\phi$ maps $R^1((0, 0, -r+1))$ to $R^1((z, 0, -s))$. As $-x^{-1}$ ranges over $\F^\times$ in (\ref{Eq13}), then $y'$ ranges over $\F^\times$.
Comparing the second and the third components in the equality in (\ref{Eq14}) we obtain that
\begin{equation}\label{zis0}
zy' = x'y'+1 \hs \text{and} \hs  zy'-1+w = s,
\end{equation}
which implies $zy' - 1 + w - s = 0$ for every $y' \in \F^\times$. As $z, w, s$ here are fixed, this implies that $z = 0$ and consequently $s = w-1$, so $\phi[0, 0, r-1] = [0, 0, w-1]$.
\end{proof}

\noindent The following is an immediate corollary of the proof above.

\begin{cor}\label{rtow}
Let $r \in \F$. If $\phi[0, 1, r] = [0, 1, w]$, then $\phi[0, 0, r-1] = [0, 0, w-1]$.
\end{cor}

Using the fact that $G$ acts on $P_0$, we now demonstrate that the action of any $\phi \in G$ on the second component of points $(0, a, r) \in P_a$ independent of $r$.

\begin{lem}\label{PaPb}
Let $\phi \in G$. Then for every $a \in \F$, there exists $b \in \F$ such that $\phi(P_a) = P_b$.
\end{lem}

 \begin{proof}
Let $r \in \F$ and $\psi = \psi_r$ belong to the stabilizer of $[0, 0, -r]$ in $G$. Since $G$ acts on $P_0$ and $\psi$ fixes $[0, 0, -r]$, then $\psi$ must fix $(0, 0, r) \in R^1([0, 0, -r])$. This implies that $\psi$ fixes
$$
F_r = \{ (x, 0, r): x \in \F^\times \} = R^1([0, 0, -r])\setminus \{ (0, 0, r) \}.
$$

In order to prove our assertion, we will use the intersection of the 2-neighborhoods of every point in $F_{r}$. Namely,
$$
N_r = \bigcap_{v \in F_{r}} R^2(v).
$$
Let $v = (x, 0, r) \in F_{r}$. Any point in $R^2(v)$ cannot have its first component equal to $x$. As $x$ ranges over $\F^\times$, then $N_r$ can only have points whose first component is 0. This immediately implies that the third component of any point in $N_r$ must be $r$. The second component can take any value in $\F$, since for any $x, a \in \F$, we have a 2-path
$$
(x, 0, r) \sim [-a/x, -a, -r] \sim (0, a, r).
$$
Thus
$$
N_r = \{ (0, a, r) : a \in \F \}.
$$
Since $\psi$ fixes $F_{r}$, then $\psi$ fixes $N_r$. Therefore, for any $a \in \F$, there exists $b \in \F$ such that $\psi(0, a,r)  = (0, b, r)$.


Now, let $\phi \in G$ and $r \in \F$. As $\phi$ fixes $L_0$, then $\phi[0, 0, -r] = [0, 0, -r']$. Therefore, $t_{r-r'}\phi[0, 0, -r] = [0, 0, -r]$, i.e. $t_{r -r'}\phi$ stabilizes $[0, 0, -r]$. Thus, from the above we have that for any $a \in \F$, there exists $b \in \F$ such that $t_{r - r'}\phi(0, a, r) = (0, b, r)$. Therefore,
$\phi(0, a, r) = (0, b, r')$. It is conceivable that $b$ may depend on both $a$, and $r$, but the following argument demonstrates this is not the case.

Recall that Lemma \ref{Dist6verts} states that $(0, a, r)$ and $(0, a, s)$ are at distance at least 6 from one another when $r \neq s$. Therefore, their images under $\phi$ must also be at distance at least 6. Suppose that
$$
\phi(0, a, r) = (0, b, r') \hs \text{and} \hs \phi(0, a, s) = (0, c, s').
$$
According to Lemma \ref{Dist6verts}(\textit{ii}), $(0, b, r')$ and $(0, c, s')$ are at distance 4 if $b \neq c$, and so we must have $b = c$. Thus $\phi(P_a) \subset P_b$ and as $|P_a| = |P_b|$, we must have $\phi(P_a) = P_b$.
\end{proof}

\begin{lem}\label{lampi23}
Let $\phi \in G$. Then $\lambda_i$ and $\pi_i$ where $i = 2, 3$ depend only on the $i$th component of a vertex. Furthermore, $\lambda_1(0) = 0$,
$\pi_2(-a) = -\lambda_2(a)$, and  $\pi_3(-r) = -\lambda_3(r)$ for all $a, r \in \F$.
\end{lem}

 \begin{proof}
The logic of the proof is as follows. First we show that $\pi_3(0, a, r) = \pi_3(r)$ for all $a, r \in \F$. Then we use this fact to demonstrate our assertion for $\lambda_2$ and $\lambda_3$. This will imply that $\lambda_1(0, a, r) = \lambda_1(0) = 0$ for all $a, r \in \F$, which in turn will allow us to prove the assertion for $\pi_2$ and $\pi_3$.

In the proof of Lemma \ref{PaPb} we demonstrated that if $\phi[0, 0, -r] = [0, 0, -r']$ (which is true by Lemma \ref{FelThm}), then for every $a \in \F$, there exists $b \in \F$ such that $\phi(0, a, r) = (0, b, r')$. Note that here $r'$ does not depend on $a$, which demonstrates that $\pi_3(0, a, r) = r' = \pi_3(r)$, for any $a \in \F$.

By Lemma \ref{PaPb}, $\pi_2(0, a, r) = \pi_2(a)$, and by the paragraph above $\pi_3(0, a, r) = \pi_3(r)$. Since $[x, -a, -r] \sim (0, a, r)$, then
$ \phi[x, -a, -r] \sim \phi(0, a, r) = (0, \pi_2(a), \pi_3(r)).$
This adjacency implies that for any $x, a, r \in \F$, $\lambda_2(x, -a, -r) = -\pi_2(a)$ and $\lambda_3(x, -a, -r) = -\pi_3(r)$. This proves the assertion about $\lambda_2$ and $\lambda_3$. Hence, from here on we can write $\lambda_2(x, a, r)$ as $\lambda_2(a)$ and $\lambda_3(x, a, r)$  as $\lambda_3(r)$.

Now we show that $\lambda_1(0, a, r) = 0$.  By Lemma \ref{Dist6verts}(\textit{ii}) We know that $[0, 0, r]$ and $[0, a, r]$ for $a \neq 0$ are at distance 4. Since $G$  acts on $L_0$ by Lemma \ref{FelThm}, then $\phi[0, 0, r] = [0, 0, \lambda_3(r)]$. Hence, $\phi[0, a, r] = [\lambda_1(0, a, r), \lambda_2(a), \lambda_3(r)]$ must be at distance 4 from $[0, 0, \lambda_3(r)]$. If $\lambda_1(0, a, r) = x' \neq 0$, then the following path shows that $\phi[0, 0, r]$ and $\phi[0, a, r]$ are at distance 2:
$$
[0, 0, \lambda_3(r)] \sim (\lambda_2(a)/x', 0, -\lambda_3(r)) \sim [x', \lambda_2(a), \lambda_3(r)].
$$
Thus $\lambda_1(0, a, r) = x' = 0$.

We now have that $\phi[0, a, r] = [0, \lambda_2(a), \lambda_3(r)]$ for any $a, r  \in \F$. As  $(x, -a, -r) \sim [0, a, r]$ for any $x \in \F$, $\phi(x, -a, -r) \sim \phi[0, a, r] = [0, \lambda_2(a), \lambda_3(r)]$. Thus by the adjacency relations,  $\pi_2(x, -a, -r) = -\lambda_2(a)$ and $\pi_3(x, -a, -r)  = -\lambda_3(r)$. Hence, from here on we can write $\pi_2(x, a, r)$ as $\pi_2(a)$ and $\pi_3(x, a, r)$  as $\pi_3(r)$. Furthermore, we obtain
\begin{equation}\label{Eq25}
\pi_3(-a) = -\lambda_3(a) \hs \text{and} \hs \pi_3(-r) = -\lambda_3(r).
\end{equation}
\end{proof}

We now show some consequences of these results. In particular, we will demonstrate that $\lambda_2(a) = a = \pi_2(a)$ for all $a \in \F_p \subset \F$.


\begin{lem}\label{Fix3rdComp}
Let $\phi \in G$. Suppose there exists an $r \in \F$ such that $\lambda_3(r) = r$. Then $\lambda_3(r - k) = r-k$ for all $k \in \F_p$.
\end{lem}

 \begin{proof}
Corollary \ref{rtow} states that if $\phi[0, 1, r] = [0, 1, r]$, then $\phi[0, 0, r-1] = [0, 0, r-1]$. But since $\lambda_3$ is only dependent on $r$, this implies that if $\lambda_3(r) = r$, then $\lambda_3(r-1) = r-1$. Therefore, by applying this iteratively, we obtain that $\lambda_3(r-k) = r-k$ for any $k \in \F_p$.
\end{proof}

\begin{lem}\label{Fix2ndComp}
Let $\phi \in G$. Then, $\lambda_2(a) = \pi_2(a) = a$ for any $a \in \F_p \subset \F$.
\end{lem}

 \begin{proof}
The entire proof rests on observing that for a fixed $b \in \F_p$, the lines $[0, 0, b]$ and $[0, 1, 0]$ have a special intersection of their 2-neighborhoods, namely, $\{ [x, b+1, b] : x \in \F^\times \}$.

Without loss of generality, we will assume that $\pi_3(0) = 0$. If not, we may consider an alternate automorphism by composing $\phi$ with $t_m$ for some $m \in \F$. Since $t_m$ has no effect on the second component, then the claimed result holds for $\phi$ if and only if it holds for $t_m\phi$.

By Lemma \ref{Fix3rdComp}, $\phi$ fixes $[0, 1, 0]$ and $[0, 0, b]$ since $b \in \F_p$. Note that
$$
R^2[0, 0, b] = \{ [x, xy, b]: x, y \in \F, x \neq 0 \}
\hspace{0.3cm} \text{and} \hspace{0.3cm}
R^2[0, 1, 0] = \{ [x, xz+1, xz] : x, z \in \F, x \neq 0\}.
$$
Let $I_b = R^2 [0, 1, 0]  \cap R^2 [0, 0, b]$. For any line in $I_b$ with first component $x$, we must have $xz + 1 = xy$
and $xz = b$. Therefore
$$
I_b = \{ [x, b+1, b] : x \in \F^\times \}.
$$
Since $\phi$ fixes both $[0, 0, b]$ and $[0, 1, 0]$, then  $\phi(I_b) = I_b$. Thus $\lambda_2(b+1) = b+1$ by Lemma \ref{lampi23}. Consequently, $\phi$ fixes $(0, -(b+1), -b)$ because $\phi$ fixes its neighborhood
$$
R^1(0, -(b+1), -b) = I_b \cup \{ [0, b+1, b] \}.
$$
Thus, by Lemma \ref{lampi23} and the above, we have $\pi_2(-b-1) = -b-1 = -\lambda_2(b+1)$. As $b$ was an arbitrary element of $\F_p \subset \F$, we obtain the claimed result.
\end{proof}

\begin{lem}\label{3rdComp}
Let $\phi \in G$ such that $\lambda_2(a) = a$ for all $a \in \F$. Then $\pi_2(a) = a$ for all $a \in \F$, and there exists $b \in \F$ such that $\lambda_3(r) = r+b = \pi_3(r)$ for all $r \in \F$.
\end{lem}

 \begin{proof}
If $\lambda_2(a) = a$ for all $a \in \F$, then $\pi_2(a) = -\lambda_2(-a) = -(-a) = a$ by (\ref{Eq25}). If $\lambda_3(0) = b$, then consider $\phi' = t_{-b}\phi$, so that $\lambda_3'(0) = 0$. By Lemma \ref{Fix3rdComp}, we have that $\lambda_3'(r) = r$ for all $r \in \F_p$. Then for $x \neq 0$ and $r \not\in \F_p$ (so $r \neq 1$),
consider the path:
\begin{equation}\label{Eq10}
[0, 0, 0] \sim (x, 0, 0) \sim [(-r+1)/x, -r+1, 0] \sim (xr/(r-1), -1, -r) \sim [0,1, r].
\end{equation}

Keeping in mind that $\lambda_2'$ is the identity map (as $t_{-b}$ does not affect the first two components) and $\lambda_3'(0) = 0$, we apply $\phi'$ to (\ref{Eq10}). Then $\phi'[0, 0, 0] = [0, 0, \lambda_3'(0)] = [0, 0, 0]$ and $\phi'[0, 1, r] = [0, 1, s]$ by Lemma \ref{FelThm}. Hence, the image of the path in (\ref{Eq10}) is:
$$
[0, 0, 0] \sim (x', 0, 0) \sim [y', -r+1, 0] \sim (z', -1, -s) \sim [0, 1, s],
$$
where the existence of $x', y', z'$ is guaranteed as $\phi'$ is an automorphism. Hence,
$$
(- r+ 1) - 1 = y' z' \hs \text{and} \hs y' z'  = -s.
$$
This implies $r= s$. Thus $\phi'[0, 1, r] = [0, 1, r]$ for every $r \in \F$. Therefore $\lambda_3'(r) = r$ for every $r \in \F$. If $\lambda_3'(r) = r$ for all $r \in \F$, then $\pi_3(r) = r$ for all $r \in \F$ again by (\ref{Eq25}). Since $\phi = t_b\phi'$, then $\lambda_3(r) =  \lambda_3'(r) + b = r-b$ and $\pi_3(r) = \pi_3'(r) - b = r+b$.
\end{proof}

\begin{lem}
Let $\phi \in G$ such that $\lambda_2(a) = a$ for all $a \in \F$. Then $\lambda_1(x) = x = \pi_1(x)$ for all $x \in \F$.
\end{lem}

 \begin{proof}
By Lemma \ref{3rdComp}, we may assume that there exists $b \in \F$ such that $\phi' = t_{-b}\phi$ has $\lambda'_2(a) = a = \pi_2'(a)$ for all $a \in \F$ and $\lambda'_3(r) = r = \pi_3'(r)$ for all $r \in \F$. We will demonstrate that as a result $\phi'$ is the identity automorphism, and therefore $\phi = t_b$.

Consider the following path for every $x, z \in \F^\times$:
$$
(0, a, r) \sim [x, -a, -r] \sim (z, zx+a, xz(zx+a)(z+zx+a+z(zx+a))+ r)
$$
As $\phi'(0, a, r) = (0, a, r)$, then applying $\phi'$ to this path yields:
$$
(0, a, r) \sim [x', -a, -r] \sim (z', z'x'+a, x'z'(z'x'+a)(z'+z'x'+a+z'(z'x'+a))+ r).
$$
Our goal is to prove that $x' = x$. This will imply that $\phi'$ fixes $[x, -a, -r]$ and therefore acts as the identity on the set of lines of $R$.

Since the second and the third components of every vertex are fixed by $\phi'$, we obtain
\begin{eqnarray*}\label{Eq8}
zx &=& z'x'\\
    x'z'(z'x'+a)(z'+z'x'+a+z'(z'x'+a)) &=& xz(zx+a)(z+zx+a+z(zx+a)).
\end{eqnarray*}
As $zx \neq 0$, then for $z \neq -a/x$, the above equation can be reduced to
\begin{equation}\label{Eq9}
 z'(1+zx+a) = z(1+zx+a).
\end{equation}
If $1 + zx + a \neq 0$, then (\ref{Eq8}) implies $z = z'$. Clearly, this inequality holds if $z \neq -(a+1)/x$. Thus $z = z'$ when $z \neq 0, -a/x, -(a+1)/x$. Therefore, for $q \geq 4$, there is at least one value of $z$ for which $z = z' \neq 0$. Since we know $zx = z'x'$, then $x = x'$. Therefore, $\phi'$ must fix $[x, -a, -r]$. Since the choice of $x$ was arbitrary, as was the choice of $a$ and $r$, then $\lambda_1(x, a, r) = x$, so that $\phi'$ acts as the identity on the lines of $R$. Since $\phi'$ fixes every line in $R$, clearly it must fix every point, so that $\phi'$ is the identity automorphism. Therefore, $\pi_1(x, a, r) = x$. Thus $\phi = t_b\phi' = t_b$ as claimed.
\end{proof}

The last sentence in the proof above implies the following corollary.

\begin{cor}
Let $\phi \in G$, and $\lambda_2(a) = a$ for all $a \in \F$, then $\phi = t_b$ for some $b \in \F$.
\end{cor}

When $q$ is prime, then $\F = \F_p$. By Lemma \ref{Fix2ndComp}, we have that $\lambda_2(a) = a$ for all $a \in \F$. Therefore, all the conditions of Corollary 4.11 are satisfied in this case, and the proof of Theorem 2.1 is complete. $\hspace{11cm} \Box$




\section{Concluding remarks} \label{sec5}


We would like to  note that for odd prime powers $q=p^e$, each of the groups  $Aut(\Gamma_i)$, $i=2,3$,  contains a cyclic subgroup of order $e$ related to the Frobenius automorphism  of the field $\F$, namely  $\phi_p:  (p_1,p_2,p_3)\mapsto (p_1^p,p_2^p,p_3^p)$ and $[l_1,l_2,l_3]) \mapsto  [l_1^p,l_2^p,l_3^p]$. This implies that $Aut(\Gamma _i)$ contains a subgroup of order $eq^{4-i}$, that is a semidirect product of $(\F^{4-i}, +)$ and $\langle \phi_p \rangle$.  In fact, this subgroup  seems to be the whole $Aut(\Gamma _i)$,  and this was verified by computer for all odd prime powers $q$, $q\le 41$.
\begin{conj} (\cite{nassau2020})   For all odd prime powers $q$,  $|Aut(R)|=eq$.
\end{conj}
We would like to conclude this paper with the following problem that is analogous to Question 1 in the case where $\F$ is infinite. Clearly,   methods used in this paper will not work in this case.  For example, if $\F=\R$ -- the field of real numbers, additive groups $(\R, +)$ and $(\R^2, +)$ are isomorphic, which  can be shown by using a Hamel basis of  the vector space $\R$  over the field of rational numbers $\Q$.
\begin{pro} Let $\F$ be an infinite field. Is there a graph $\Gamma_3 = \Gamma_\F (f(p_1,l_1), g(p_1,p_2,l_1))$ that is not isomorphic to  a graph $\Gamma_2 = \Gamma_\F (f(p_1,l_1), h(p_1,l_1))$, where $f$, $g$ and $h$ are polynomial functions?
Does the answer to this question change if we allow the  functions $f$, $g$ and $h$ to be arbitrary functions on $\F$?
\end{pro}

Let $q = p^e$ be an odd prime power, $q \geq 5$. The graph $\Gamma = \Gamma_{\F_q}(p_1\ell_1)$ is sometimes referred to as the bi-affine part of the point-line incidence graph of the classical projective plane of order $q$. The graph $\Gamma_3 = \Gamma_{\F_q}(p_1\ell_1, g(p_1, \ell_1, p_2))$ is a $q$-lift of $\Gamma_{\F_q}(p_1\ell_1)$. Sometimes $|Aut(\Gamma_3)|$ is larger than $|Aut(\Gamma)| = 2eq^3(q-1)^2$, see \cite{Vig2002}. For example, when $g = p_1\ell_2$, $|Aut(\Gamma_3)| = eq^4(q-1)^2$, see \cite{Vig2002}. In this paper we showed that when $q = p$ and $p \equiv 1 \pmod{3}$, $|Aut(R)| = p$, demonstrating that at other times, a $q$-lift of $\Gamma$ can have an automorphism group with much smaller order. This motivates the following problem.

\begin{pro}
Let $\F_q$ be the finite field of $q$ elements. Describe all possible groups $Aut(\Gamma_{\F_q}(p_1\ell_1, g(p_1, \ell_1, p_2)))$.

\end{pro}

\section{Acknowledgement}  This work was partially supported by the Simons Foundation Award ID: 426092 and the National Science Foundation Grant: 1855723.

\newpage

\section{Appendix}

\centerline{\large{{\bf  Maple Code for finding the 3-neighborhood of a vertex of $R$}}}
\bigskip

{\small
\noindent\verb!#This program builds a path from a vertex of graph R using the first!

\noindent\verb!#components of the vertices in the path.!

\medskip

\noindent\verb!restart;!

\medskip

\noindent\verb!#Given a line, this procedure finds the neighbor of the line with given first component.!

\medskip

\noindent\verb!nl := proc (line::list, b::algebraic)!
\verb!local nbl, f2, f3, p, l, l1, l2, l3, p1, p2, p3;! \newline
\verb!l1 := line[1]; l2 := line[2]; l3 := line[3];! \newline
\verb!f2 := (x1, y1)! $\rightarrow$  \verb!x1*y1;! \newline
\verb!f3 := (x1, x2, y1, y2)! $\rightarrow$ \verb!x1*x2*l1*(x1+x2+x1*x2);! \newline
\noindent\verb!nbl := solve({p2+l2 = f2(p1, l1), p3+l3 = f3(p1, p2, l1, l2), p1 = b}, {p1, p2, p3});! \newline
\noindent\verb!p := [subs(nbl, p1), subs(nbl, p2), subs(nbl, p3)];! \newline
\noindent\verb!end:! \\

\noindent\verb!#Given a point, this procedure finds the neighbor of the point with given first component.!\\

\noindent\verb!np := proc (point::list, a::algebraic)!
\verb!local nbp, f2, f3, p, l, l1, l2, l3, p1, p2, p3;! \newline
\verb!p1 := point[1]; p2 := point[2]; p3 := point[3];! \newline
\verb!f2 := (x1, y1)! $\rightarrow$  \verb!x1*y1;! \newline
\verb!f3 := proc (x1, x2, y1, y2)! $\rightarrow$ \verb!x1*x2*l1*(x1+x2+x1*x2);!\newline
\noindent\verb!nbp := solve({p2+l2 = f2(p1, l1), p3+l3 = f3(p1, p2, l1, l2), l1 = a}, {l1, l2, l3});! \newline
\noindent\verb!l := [subs(nbp, l1), subs(nbp, l2), subs(nbp, l3)];! \newline
\verb!end:!

\medskip

\noindent\verb!#A 3-Path the begins at line [A, B, 0]! \newline
\verb!#!

\medskip

\noindent\verb!L1 := [A, B, 0];! \newline
\verb!P1 := nl(L1, a);! \newline
\verb!L2 := factor(np(P1, x));!  \newline
\verb!P2 := factor(nl(L2, b));! \newline
\verb!print();! \\

\noindent\verb!#A 3-Path the begins at point (A, B, 0)! \newline
\verb!#!

\medskip

\noindent\verb!#P1 := [A, B, 0];! \newline
\verb!#L1 := np(P1, x);! \newline
\verb!#P2 := factor(nl(L1, a));!  \newline
\verb!#L2 := factor(np(P2, y));! \newline
\verb!#print();!

\medskip

\noindent\verb!#PAB below stands for P_{A, B}(b, c; a)! \newline
\verb!#!

\medskip

\noindent\verb!PAB := simplify(subs({x = (-A*a+B+c)/(b-a)}, P2));! \\

\noindent\verb!#3-neighborhoods of lines [0, 0, 0] and [0, 1, 0]! \newline
\verb!#!

\medskip

\verb!subs({A=0,B=0},PAB);!  \newline
\verb!subs({A = 0, B = 1}, PAB);! \newline
}

\end{document}